\documentclass[a4paper,10pt]{article}
\usepackage{amssymb,amsmath,amsthm}
\usepackage{fullpage}
\usepackage{hyperref}
\usepackage{eucal}
\usepackage{mathrsfs}
\usepackage{color}
\usepackage{stackrel}
\usepackage{graphicx}
\newcommand{\ie}{\emph{i.e.}}
\newcommand{\eg}{\emph{e.g.}}

\newcommand{\Real}{\mathbb{R}}

\newcommand{\Nat}{\mathbb{N}}

\newcommand{\Sphere}{\mathbb{S}}

\newcommand{\dom}{\mathop{\mathrm{dom}}\nolimits}

\newcommand{\sii}{L^2}

\newcommand{\der}{\mathrm{d}}

\newtheorem{Theorem}{Theorem}
\newtheorem{Lemma}{Lemma}
\newtheorem{Proposition}{Proposition}

\theoremstyle{definition}
\newtheorem{Remark}{Remark}

\def\OMIT#1{}
\usepackage[normalem]{ulem}

\definecolor{DarkGreen}{rgb}{0,0.5,0.1} % David

\newcommand\soutD{\bgroup\markoverwith
{\textcolor{DarkGreen}{\rule[.5ex]{2pt}{1pt}}}\ULon}
\newcommand\soutP{\bgroup\markoverwith
{\textcolor{blue}{\rule[.5ex]{2pt}{1pt}}}\ULon}
\newcommand{\Hm}[1]{\leavevmode{\marginpar{\tiny%
$\hbox to 0mm{\hspace*{-0.5mm}$\leftarrow$\hss}%
\vcenter{\vrule depth 0.1mm height 0.1mm width \the\marginparwidth}%
\hbox to
0mm{\hss$\rightarrow$\hspace*{-0.5mm}}$\\\relax\raggedright #1}}}

\begin{document}
% 
%-------%
% TITLE %
%-------%
%------------------------------------------%
%------------------------------------------%
\title{\textbf{\Large Sharp estimates for the Laplacian torsional rigidity with negative Robin boundary conditions}}
\author{Nunzia Gavitone,$^a$ \
David Krej\v{c}i\v{r}{\'\i}k\,$^b$
\ and \ Gloria Paoli\,$^a$}
\date{\small 
\begin{quote}
\emph{
\begin{itemize}
\item[$a)$]  
Dipartimento di Matematica e Applicazioni ``Renato Caccioppoli'', 
Universit\`a degli Studi di Napoli Federico II, 
Via Cintia, Monte S. Angelo, 80126 Napoli, Italy;
nunzia.gavitone@unina.it, gloria.paoli@unina.it.
\item[$b)$] 
Department of Mathematics, Faculty of Nuclear Sciences and 
Physical Engineering, Czech Technical University in Prague, 
Trojanova 13, 12000 Prague 2, Czechia;
david.krejcirik@fjfi.cvut.cz.%
\end{itemize}
}
\end{quote}
{\small 13 January 2026}}
\maketitle
%
%------------------------------------------%
%------------------------------------------%

%
\begin{abstract}
\noindent
Motivated by pioneering works of Bandle and Wagner
\cite{BW-shape,BW-sign}, given a bounded Lipschitz domain 
$\Omega \subset \mathbb R^d$ with $d\ge3$,
we consider the Robin-Laplacian torsional rigidity $\tau_\alpha(\Omega)$ 
with negative boundary parameter~$\alpha$ 
and we show that sharp inequalities for $\tau_\alpha(\Omega)$ hold if $|\alpha|$ is small enough.
In particular, we prove that, if $|\alpha|$ is smaller than 
the first non-trivial Steklov-Laplacian eigenvalue, then  the ball maximises $\tau_\alpha(\Omega)$ %among all simply connected planar domains of fixed perimeter
among all convex domains under perimeter or  volume constraints. 
This solves an open problem raised in~\cite{BW-sign}. 
We also prove the result in the planar case 
among simply connected sets and under perimeter constraint.
%

%
%\bigskip
%\begin{itemize}
%\item[\textbf{Keywords:}]
%\item[\textbf{MSC (2010):}]
%\end{itemize}
%
\end{abstract}
%
%------------------------------------------%
%------------------------------------------%
 
%---------------------%
\section{Introduction and main results}
%---------------------%
%
The expected lifetime of the Brownian motion 
killed on the boundary of a domain of given volume 
is the longest in the ball.
This is the celebrated Saint-Venant inequality~\cite{Polya-Szego}.
Another physical interpretation of this result is that 
the cylindrical beam is the most resistive against torsional deformations,
among all elastic beams of cross-section of given area.
 
What happens for the Brownian motion with other boundary conditions?
Is the ball still the optimal geometry?
This physical curiosity is behind considering
the boundary value problem~\cite{Campanino-DelGrosso_1976} 
\begin{equation}\label{capacity}
\left\{
\begin{aligned}
  -\Delta u &=1 
  && \mbox{in} \quad \Omega \,, \\
  \frac{\partial u}{\partial n} + \alpha \;\! u &=0  
  && \mbox{on} \quad \partial\Omega \,, 
\end{aligned}  
\right.
\end{equation}
where $\alpha \in \Real$ and~$n$ is the unit outward normal 
to a bounded Lipschitz domain $\Omega \subset \Real^d$ with $d \geq 1$.
The case of positive (respectively, negative)~$\alpha$ corresponds
to Brownian particles absorbed (respectively, created) on the boundary.
The Neumann ($\alpha=0$) and Dirichlet (formally, $\alpha=\infty$) 
realisations corresponds to the reflecting and killing 
boundary conditions, respectively.
The expected lifetime of the Brownian motion
or the \emph{torsional rigidity} is (up to a physical constant)
the quantity \cite{Banuelos-Berg-Carroll_2002,Carroll_2001}
$$
  \tau_\alpha(\Omega) := \int_\Omega u(x) \, \der x
  \,
$$ 
where~$u$ is the solution of~\eqref{capacity}.
In addition to the aforementioned direct physical interpretations,
the torsion function~$u$ plays an important role 
as a convenient gauge choice for magnetic problems
(see, \eg, \cite{Barbaroux-LeTreust-Raymond-Stockmeyer_2021,
Colbois-Lena-Provenzano-Savo_2023,Kachmar-Lotoreichik_2024}
for recent applications). 

In recent years many mathematicians studied shape optimisations problems related to $\tau_\alpha(\Omega)$ under several geometric constraints (see for instance \cite{ANT,amato2025quantitative, BW-sign,Bandle-Wagner, BW-shape, Bucur-Giacomini_2015b, gavitone2025serrin, Vikulova_2022}).
In particular, we stress that the behaviour and the properties of $\tau_\alpha(\Omega)$ depend on  the sign of $\alpha$.

In the case of absorption (\ie~$\alpha>0$), 
Bucur and Giacomini in~\cite{Bucur-Giacomini_2015b} 
established the following global optimality of the ball under volume constraint.

\begin{Theorem}[Bucur and Giacomini~\cite{Bucur-Giacomini_2015b}]
Let $d \geq 2$ and $\Omega\subset\mathbb{R}^d$ be a bounded and Lipschitz set.
If $\alpha > 0$, then
\begin{equation}\label{Saint}
  \max_{|\Omega|=c} \tau_\alpha(\Omega) = \tau_\alpha(\Omega^*)
  \,,
\end{equation}
where $\Omega^*$ is the ball of the same volume as~$\Omega$
fixed by a given positive constant~$c$. 
\end{Theorem}

The classical Saint-Venant inequality is a special case of~\eqref{Saint} 
when the Robin boundary condition in~\eqref{capacity}\
is replaced by the Dirichlet one ($\alpha=\infty$).
In this case, we denote the torsional rigidity by~$\tau_D(\Omega)$.

The aim of the present paper is to investigate what  happens when $\alpha$ is negative. Indeed, as far as the case $\alpha<0$ is concerned, up to our knowledge, there exists no global optimality result under measure constraint in all dimensions. 
In this setting  by using a shape derivative approach, in~\cite{BW-shape}, Bandle and Wagner proved that the ball is a local minimiser
for $\tau_\alpha(\Omega)$
if $|\alpha|$ is small enough,
among all domains of fixed volume and in any dimension. 
Thereafter, in~\cite{BW-sign}, 
Bandle and Wagner (see also their book~\cite[Sec.~13]{Bandle-Wagner})
showed the global optimality of the disk in the planar case if $|\alpha|$ is small enough. The smallness depends on the geometry of~$\Omega$ and more specifically,
it is quantified by the Steklov eigenvalue problem
\begin{equation}\label{Steklov} 
\left\{
\begin{aligned}
  -\Delta \phi &=0 
  && \mbox{in} \quad \Omega \,, \\
  \frac{\partial \phi}{\partial n} &= \sigma \;\! \phi 
  && \mbox{on} \quad \partial\Omega \,,
\end{aligned}  
\right.
\end{equation}
which is known to admit an infinite sequence of discrete eigenvalues
(see \cite[Sec.~5]{Henrot2} or
\cite[Sec.~7]{Levitin-Mangoubi-Polterovich})
$$
  0 = \sigma_0(\Omega) 
  < \sigma_1(\Omega)
  \leq \sigma_2(\Omega)
  \leq \dots \to +\infty
  \,.
$$
More precisely, the global optimisation result 
proved in~\cite{BW-sign} reads as follows.%
\footnote{Actually the theorem is stated in~\cite{BW-sign} 
without requiring that~$\Omega$ is simply connected,
however, this hypothesis seems necessary 
for the proof given there
(based on Proposition~\ref{Prop.Payne} below).}
\begin{Theorem}[Bandle and Wagner~\cite{BW-sign}]\label{Thm.Bandle}
Let $d = 2$ and assume that~$\Omega$ is simply connected.
If $\alpha \in (-\sigma_1(\Omega),0)$, then
\begin{equation}\label{Bandle}
  \min_{|\Omega|=c} \tau_\alpha(\Omega) = \tau_\alpha(\Omega^*)
  \,,
\end{equation}
where $\Omega^*$ is the disk of the same area as~$\Omega$
fixed by a given positive constant~$c$. 
\end{Theorem}

Let us stress that the authors in \cite{BW-sign} do not study the case of  perimeter constraint and, moreover, they state as an open problem the extension of the previous result to higher dimensions. 
The objective of this paper is precisely to prove
 Theorem~\ref{Thm.Bandle} in higher dimensions 
and/or under the alternative isoperimetric constraint, 
the former solving the open problem raised in \cite{BW-sign}.

Our motivation to study the perimeter constraint comes from the following Robin eigenvalue problem
\begin{equation}\label{eigenvalue}
\left\{
\begin{aligned}
  -\Delta u &=\lambda u  
  && \mbox{in} \quad \Omega \,, \\
  \frac{\partial u}{\partial n} + \alpha \;\! u &=0  
  && \mbox{on} \quad \partial\Omega \,,
\end{aligned}  
\right.
\end{equation}
where $\lambda$ is a real number and $\alpha<0.$
Indeed, for the lowest eigenvalue of \eqref{eigenvalue}  it is known that the ball equally stops to be the optimal set 
under the isochoric constraint if $|\alpha|$ is large (see~\cite{FK7}), while this transition does not take place
if the isoperimetric constraint is assumed
(see  \cite{AFK,BFNT,Vikulova_2022}). 
In this framework, our first result reads as follows.

\begin{Theorem}\label{Thm.2D}
Let $d = 2$ and assume that~$\Omega$ is simply connected.
If $\alpha \in (-\sigma_1(\Omega),0)$, then
\begin{equation} 
  \min_{|\partial\Omega|=c} \tau_\alpha(\Omega) = \tau_\alpha(\Omega^*)
  \,,
\end{equation}
where $\Omega^*$ is the disk of the same perimeter as~$\Omega$ 
fixed by a given positive constant~$c$. 
\end{Theorem}

As regards the higher-dimensional case, 
our main result, which generalises Theorems~\ref{Thm.Bandle} and~\ref{Thm.2D}, is the following.

\begin{Theorem}\label{Thm.all}
Let $d \geq 3$ and assume that $\Omega$ is convex. 
If $\alpha \in (-\sigma_1(\Omega),0)$,
then
\begin{equation} 
  \min_{|\Omega|=c} \tau_\alpha(\Omega) = \tau_\alpha(\Omega^*)
  \qquad
  \mbox{(respectively, }
  \min_{|\partial\Omega|=c} \tau_\alpha(\Omega) = \tau_\alpha(\Omega^*)
  \mbox{)}
  \,,
\end{equation}
where $\Omega^*$ is the ball of the same volume 
(respectively, perimeter) as~$\Omega$ 
fixed by a given positive constant~$c$. 
If $d=3$, the convexity can be replaced by the hypothesis 
that~$\Omega$ is axiconvex 
with~$\partial\Omega$ being homeomorphic to the sphere.
\end{Theorem}

We recall  that a three-dimensional Euclidean set is called axiconvex 
if it is rotationally invariant around an axis 
and its intersection with any plane orthogonal to the symmetry
axis is either a disk or empty set 
\cite{Dalphin-Henrot-Masnou-Takahashi_2016}.
 
The proof of Theorem~\ref{Thm.all} 
is based on Theorem~\ref{Prop.new} below
stating that the Dirichlet torsional rigidity 
of any convex domain in all dimensions $d \geq 2$ is bounded from below
by the Dirichlet--Neumann torsional rigidity of a spherical shell
of the same volume. 
To the best of our knowledge, this result is new for $d \geq 4$.
(The case $d=3$ can be deduced from~\cite{Vikulova_2022}.)
Our proof is based on the extension of the method of parallel coordinates
following the ideas of~\cite{BFNT}.

The organisation of this paper is as follows.
In Section~\ref{Sec.pre}, we specify how to frame~\eqref{capacity}
operator-theoretically, 
present explicit solutions for balls,
state the main ingredient in our proofs
(replacing the lack of a suitable variational characterisation
if~$\alpha$ is negative)
and recall the Steiner formula for convex sets. 
The proofs of Theorems~\ref{Thm.2D} and~\ref{Thm.all}
are given in Sections~\ref{Sec.2D} and~\ref{Sec.all}.

%----------------------%
\section{Preliminary results}\label{Sec.pre}
%----------------------%
%

\subsection{Notations and basic facts} 

Throughout this paper, we assume that~$\Omega$
is a non-empty bounded Lipschitz domain in~$\Real^d$ with $d \geq 2$
(a domain is an open connected set).
The $d$-dimensional Lebesgue measure of~$\Omega$ 
and the $(d-1)$-dimensional Hausdorff measure of the boundary~$\partial\Omega$
are denoted by~$|\Omega|$ and~$|\partial\Omega|$, respectively. 
Moreover, we denote by $B_R$   the open ball of~$\Real^d$ with  radius~$R>0$
centred at the origin, and the spherical shell by $A_{R_1,R_2} := B_{R_2} \setminus \overline{B_{R_1}}$.
We also recall that the inradius of a set $\Omega \subset \mathbb R^d$ is defined by
\begin{equation} \label{dist}
 r_{\Omega} := \sup_{x \in \Omega} \rho(x),
 \qquad \mbox{where} \qquad
 \rho(x) := \inf_{y \in \partial \Omega} |x-y|
\end{equation}
is the  distance function from the boundary of $\Omega$.

In what follows we recall some properties of convex bodies, which will be useful in the sequel (see as reference \cite{schneider}). 
Let $\emptyset\neq K\subseteq\mathbb{R}^d$ be  an open, 
pre-compact and convex set. 
We define the \emph{outer parallel body} of~$K$ 
at distance $\rho \geq 0$ as the  Minkowski sum
$$ 
K+\rho B_1=\{ x+\rho y\in\mathbb{R}^d:\  x\in K,\;y\in B_1 \}.
$$ 
The Steiner formula asserts that
\begin{equation}\label{general_steiner}
|K+\rho B_1|=\sum_{i=0}^{d}\binom{d}{i} \, W_i(K) \, \rho^i
\,, 
\end{equation}
where the coefficients $W_i(\Omega)$ are usually called \emph{quermassintegrals}.
Consequently,
\begin{equation}\label{second W}
\lim\limits_{\rho \to 0} 
\dfrac{|\partial(K+\rho B_1)|-|\partial K|}{\rho}
=d(d-1) W_2(K)
\,.
\end{equation}
For smooth~$K$, one has $W_2(K) = \frac{1}{d} \int_{\partial K} H$,
where $H := \frac{1}{d-1} (\kappa_1 + \dots + \kappa_{d-1})$ 
denotes the mean curvature of~$\partial K$, 
with the convention that the principal curvatures 
$\kappa_1,\dots,\kappa_{d-1}$ are non-negative if~$K$ is convex.

The Aleksandrov--Fenchel inequalities state 
\begin{equation}\label{aleksandrov-fenchel}
	\left(\dfrac{W_j(K)}{|B_1|}\right)^{\frac{1}{d-j}}
	\geq \left(\dfrac{W_i(K)}{|B_1|}\right)^{\frac{1}{d-i}},
	\qquad\mbox{for}\qquad
	0\leq i<j<n \,,
\end{equation}
with equality if, and only if, $K$~is a ball.  
In the case  $i=0$ and $j=1$, 
we obtain the classical isoperimetric inequality
\begin{equation*}%\label{classical_iso}
|\partial K|^{\frac{d}{d-1}}
\geq d^{\frac{d}{d-1}} \, |B_1|^{\frac{1}{d-1}} \, |K|
\,. 
\end{equation*}
When $i=1$ and $j=2$,  \eqref{aleksandrov-fenchel} reads
\begin{equation}\label{iso_W_2}
W_2(K)\geq 
d^{-\frac{d-2}{d-1}} \, |B_1|^{\frac{1}{d-1}} \, |\partial K|^{\frac{d-2}{d-1}}
\,.
\end{equation}

For every $t \in[0,r_{\Omega}]$, we consider the following super-level sets of the distance function 
\begin{equation*}%\label{def_omega_t}
\Omega_{t}:=\{ x\in \Omega:\ \rho(x)>t  \} 
\,,
\end{equation*}
we recall the following two lemmata, whose proofs can be found, for instance, 
in~\cite{BNT} and~\cite{BFNT}.
\begin{Lemma} 
	Let $\Omega$ be a bounded, convex, open set in $\mathbb{R}^d$. Then, for almost every $t\in(0,r_{\Omega} )$, we have
	\begin{equation*}%\label{derivative_perimeter}
		-\dfrac{\der}{\der t} 
		|\partial\Omega_{t}|
		\geq d(d-1) \, W_2(\Omega_{t})
		\,,
	\end{equation*}
	with equality holding if $\Omega$ is a ball.
\end{Lemma}
\begin{Lemma}
\label{lem_der_per_e}
In addition,
	let $f:[0,+\infty)\to[0,+\infty) $  be a non-decreasing $C^1$ function.   %and let $\tilde f:[0,+\infty)\to[0,+\infty) $  a non increasing  $C^1$ function. 
    Define  
	\begin{equation*}
	u(x):=f(\rho(x))
	\qquad \mbox{and} \qquad 
	E_{t}:=\{x\in \Omega\;:\; u(x)>t\}. 
	\end{equation*}
	Then,
	\begin{equation}\label{derivata_composta_super1}
		-\dfrac{\der}{\der t}
		|\partial E_{t}|
		\geq d (d-1) \, \dfrac{W_2(E_{t})}{|D u|_{u=t}}
		\,.
	\end{equation}
	%and 
		%\begin{equation}\label{derivata_composta_sub1}
	%	\dfrac{d}{dt}P(\tilde E_{0,t})\geq n (n-1)\dfrac{W_2(\tilde E_{0,t})}{|D \tilde u|_{\tilde u=t}}.
	%\end{equation}
\end{Lemma}

\subsection{ Operator-theoretic framework and the key lemma
}\label{Sec.operator}
Let $-\Delta_\alpha^\Omega$ denote the Robin Laplacian, 
\ie~the self-adjoint operator in $\sii(\Omega)$
associated with the quadratic form 
\begin{equation}\label{Laplace}
  \delta_\alpha^\Omega[\psi] 
  := \|\nabla\psi\|_{\sii(\Omega)}^2 
  + \alpha \, \|\psi\|_{\sii(\partial\Omega)}^2
  \,, \qquad
  \dom(\delta_\alpha^\Omega) := W^{1,2}(\Omega)
  \,,
\end{equation}
where the boundary trace of~$\psi$ is denoted by the same symbol.
If~$\alpha$ is such that $0 \not\in \sigma(-\Delta_\alpha^\Omega)$,
there is a unique solution to~\eqref{capacity} given by 
\begin{equation}\label{solution} 
  u = (-\Delta_\alpha^\Omega)^{-1} 1
\end{equation}
This is true for all positive~$\alpha$.
If~$\alpha$ is non-positive, however, $0 \in \sigma(-\Delta_\alpha^\Omega)$
if, and only if, $\alpha = -\sigma_k(\Omega)$ 
for some $k \in \Nat := \{0,1,\dots\}$,
where $(\sigma_k(\Omega))_{k \in \Nat}$
is the sequence of Steklov eigenvalues
(repeated according to multiplicities)
of the problem~\eqref{Steklov}.
We recall that  
$
 \{\sigma_k(B_R)\}_{k\in\Nat}
 = \{\frac{k}{R}\}_{k\in\Nat}
$. 

\begin{Remark}
In~\cite{Bucur-Fragala_2019}, Bucur and Fragal\`{a} consider
an alternative ``torsional rigidity'' $\tilde{\tau}_\alpha(\Omega)$
defined variationally via
\begin{equation}\label{minimax} 
  \tilde{\tau}_\alpha(\Omega)^{-1} :=
  \inf_{\stackrel[\psi\not=0]{}{\psi\in W^{1,2}(\Omega)}} 
  \frac{\displaystyle
  \int_\Omega |\nabla\psi|^2 
  + \alpha \int_{\partial\Omega} |\psi|^2 }{\displaystyle
  \left(\int_\Omega |\psi|\right)^2} 
  \,.
\end{equation}
(For $\alpha=\infty$ the boundary term is disregarded
and the space of test functions is taken to be $W_0^{1,2}(\Omega)$.)
If $\alpha \in (0,\infty) \cup \{\infty\}$,
the two definitions coincide, \ie\ 
$\tilde{\tau}_\alpha(\Omega) = \tau_\alpha(\Omega)$.
In this case, the absolute value in the denominator can be removed
and~\eqref{capacity} represents the Euler--Lagrange equation
associated with~\eqref{minimax};
in fact, the solution to~\eqref{capacity} is necessarily positive. 
If~$\alpha$ is negative, however, 
$\tilde{\tau}_\alpha(\Omega) \not= \tau_\alpha(\Omega)$ in general,
for the solution to~\eqref{capacity} does not necessarily 
have a constant sign in~$\Omega$.
In this paper, we exclusively restrict to the physically motivated 
definition via~\eqref{capacity}.
We use~\eqref{minimax} only as a technical tool 
in Theorem~\ref{Prop.new} concerned with the Dirichlet case.
\end{Remark}
The boundary value problem~\eqref{capacity} can be interpreted as 
the Euler--Lagrange equation of the energy functional 
$E_\alpha^\Omega[\psi] := \delta_\alpha^\Omega[\psi] - 2 \int_\Omega \psi$
with $\dom(E_\alpha^\Omega) := W^{1,2}(\Omega)$.
Obviously, $E_\alpha^\Omega[u] = -\tau_\alpha(\Omega)$.
Because of the absence of a suitable variational characterisation 
of $\tau_\alpha(\Omega)$ whenever~$\alpha$ is negative,
the following estimate is a crucial ingredient in our proofs.
\begin{Lemma}[{Bandle and Wagner~\cite[Lem.~2]{BW-sign}}]\label{Lem.core}
If $\alpha \in (-\sigma_1(\Omega), 0)$, then
\begin{equation}\label{core} 
  \tau_\alpha(\Omega) \geq \tau_D(\Omega) 
  + \frac{|\Omega|^2}{\alpha |\partial\Omega|} 
  \,.
\end{equation}
The equality occurs for the ball.
\end{Lemma}

The proof of~\cite[Lem.~2]{BW-sign} is based on writing~$u$ 
as a sum of the Dirichlet torsional rigidity 
and a harmonic function,
the latter being expanded in a series of the Steklov eigenfunctions.

It is remarkable that the proof of Theorem~\ref{Thm.Bandle}
due to~\cite{BW-sign}
as well as of our present results can be reduced 
to the crude inequality~\eqref{core} 
and to estimating the Dirichlet torsional rigidity $\tau_D(\Omega)$ from below.

\subsection{The radial cases}
\subsubsection{The Robin torsional rigidity of a ball}
It is straightforward to check that the radial solution 
of~\eqref{capacity} for $\Omega = B_R$ and $\alpha\not=0$ reads
\begin{equation*} 
  u(x) = \frac{1}{2d} \left(R^2-|x|^2\right) + \frac{R}{\alpha d} \,.
\end{equation*}
It coincides with~\eqref{solution} provided that 
$
 -\alpha \not\in \{\sigma_k(B_R)\}_{k\in\Nat}
$.
We observe that~$u$ generally does not have a constant sign in~$B_R$.
More specifically, 
$u$~is positive in~$B_R$ if $\alpha>0$,
it is negative in~$B_R$ if 
$\alpha \in [-\frac{2}{R},\frac{1}{R})$ 
and it changes sign in~$B_R$ whenever $\alpha < -\frac{2}{R}$.

Under the condition $-\alpha \not\in \{\sigma_k(B_R)\}_{k\in\Nat}$,
we verify that
\begin{equation}\label{ball} 
  \frac{\tau_\alpha(B_R)}{|\Sphere^{d-1}|} = 
  \frac{R^{d+2}}{d^2(d+2)} + \frac{R^{d+1}}{\alpha d^2}
  \,.
\end{equation}
It follows that
\begin{equation*} 
  \tau_\alpha(B_R) = 0
  \qquad \Longleftrightarrow \qquad 
  \alpha = - \frac{d+2}{R} =: \alpha_* 
  \,.
\end{equation*}
Moreover, $\int_{B_R} u < 0$ if $\alpha \in (\alpha^*,0)$ 
and $\int_{B_R} u > 0$ 
if $\alpha \in (-\infty,\alpha^*) \cup (0,+\infty)$. 
Note that $\alpha_* < -\sigma_1(B_R)$.

\subsubsection{Dirichlet--Neumann torsional rigidity of a spherical shell}
Finally, we consider the following Dirichlet--Neumann torsion problem in the spherical shell:
\begin{equation}\label{capacity.annulus}
\left\{
\begin{aligned}
  -\Delta u &=1 
  && \mbox{in} \quad A_{R_1,R_2} \,, \\
  u &=0  
  && \mbox{on} \quad \partial B_{R_2} \,, \\
  \frac{\partial u}{\partial n} &=0  
  && \mbox{on} \quad \partial B_{R_1} \,,
\end{aligned}  
\right.
\end{equation}
Denote $\tau_{DN}(A_{R_1,R_2}) := \int_{A_{R_1,R_2}} u$
with~$u$ being the solution of the problem \eqref{capacity.annulus}.
By simple computations, we have 
\begin{equation}\label{torsion_annulus}
  \tau_{DN}(A_{R_1,R_2}) =
  \begin{cases}
  \displaystyle
\frac{\pi}{8} \left[4 R_1^4 \log \left(\frac{R_2}{R_1}\right)
+3 R_1^4-4 R_1^2 R_2^2+R_2^4\right] 
  & \mbox{if} \quad d=2,
\\
   |\Sphere^{d-1}|
   \left(
   \dfrac{R_2^{d+2}}{d^2(d+2)}+\dfrac{R_1^{d+2}}{(d^2-4)}
-\dfrac{R_1^d R_2^2}{d^2}-\dfrac{R_1^{2d}R_2^{2-d}}{d^2(d-2)}
  \right)
  & \mbox{if} \quad d \geq 4.
  \end{cases}
\end{equation}
%

%--------------------------------------------------------%
\section{The lower dimensional cases \texorpdfstring{$d=2,3$}{d=2,3} }\label{Sec.2D}
%--------------------------------------------------------%

In this Section we study 
separately our main theorems 
in dimension $d=2$ and $d=3$.

\subsection{The planar case}
We start with the planar case by proving our Theorem \ref{Thm.2D}
with the isoperimetric constraint.
As in the proof of the quoted Theorem \ref{Thm.Bandle} contained in~\cite{BW-sign}, the main ingredient
is the following estimate of the Dirichlet torsional rigidity
due to Payne and Weinberger (see~\cite{Payne-Weinberger_1961}).
\begin{Proposition}[Payne and Weinberger \cite{Payne-Weinberger_1961}]\label{Prop.Payne}
Let $d=2$ and assume that~$\Omega$ is simply connected.  
Then
$$
  \tau_D(\Omega) \geq \tau_{DN}(A_{R_1,R_2}) 
  \,
$$
where $A_{R_1,R_2}$
is the annulus of radii such that $|A_{R_1,R_2}| = |\Omega|$
and $|\partial B_{R_2}| = |\partial\Omega|$; explicitly,
$$
  R_1 := \frac{\sqrt{|\partial\Omega|^2-4\pi|\Omega|}}{2\pi}
  \qquad \mbox{and} \qquad
  R_2 := \frac{|\partial\Omega|}{2\pi}
$$
\end{Proposition}

Note that~$R_1$ is well defined by the isoperimetric inequality.
The proof of Proposition~\ref{Prop.Payne} is based on 
the method of parallel coordinates,
where functions with level lines parallel to the boundary~$\partial\Omega$
are chosen as trial functions in the variational characterisation 
of~$\tau_D(\Omega)$. Then we can give the proof of Theorem \ref{Thm.2D}
\begin{proof}[Proof of Theorem \ref{Thm.2D}]

    To achieve Theorem~\ref{Thm.2D} via Lemma~\ref{Lem.core}
we use  the explicit formulation of $\tau_{DN}(A_{R_1,R_2})$
given
in~\eqref{torsion_annulus},
and we compare it  with the torsional rigidity~$\tau_D(B_R)$
of a disk~$B_R$ such that $|\partial B_R| = |\partial\Omega|$, 
\ie\ $R=R_2$
(the explicit formula can be deduced from~\eqref{ball}
after sending $\alpha \to \infty$). 
Defining $t := R_1 / R_2 \in [0,1)$ and using Lemma~\ref{Lem.core},
we get
\begin{equation}\label{rhs}
  \tau_\alpha(\Omega) - \tau_\alpha(B_R)
  \geq 
  \pi R^4 \left(
  -\frac{1}{2} t^4 \log(t) 
  + \frac{3}{8} t^4 - \frac{1}{2} t^2 
  - \frac{(1-t^2)^2-1}{2 \alpha R} 
  \right)
  .
\end{equation}
Assuming $t \in (0,1)$ (the case $t=0$ is trivial)
and requiring that the right-hand side is non-negative 
lead to the following smallness condition on~$\alpha$:
$$
  - 2 \alpha R \leq \frac{(1-t^2)^2-1}{-\frac{1}{2} t^4 \log(t) 
  + \frac{3}{8} t^4 - \frac{1}{2} t^2}
  \,.
$$
Note that the denominator is non-zero (in fact, negative)
for every $t \in (0,1)$.
What is more, the right-hand side is bounded from below by~$4$.
Consequently, the right-hand side of~\eqref{rhs}
is non-negative provided that  
$ 
  \alpha \geq -\frac{2}{R} 
$.
In addition, the hypothesis of Lemma~\ref{Lem.core} requires
$
  \alpha > -\sigma_1(\Omega) 
$.
By the Weinstock inequality, 
$
  \sigma_1(\Omega) \leq \sigma_1(B_R) = \frac{1}{R}
$. 
This concludes the proof of Theorem~\ref{Thm.2D}.
 
  \end{proof}
%--------------------------------------------------------%
\subsection{The three-dimensional case }\label{Sec.3D}
%--------------------------------------------------------%
%
Now we proceed with the proof of Theorem~\ref{Thm.all}
in the special case of dimension $d=3$.
It was generally accepted that the method of parallel coordinates 
leading to Proposition~\ref{Prop.Payne} is not easy to extend 
to higher dimensions.
Recently, however, Vikulova~\cite{Vikulova_2022} 
has managed to apply it in three dimensions.
Her approach leads to the following result
(though it is not explicitly stated in~\cite{Vikulova_2022}).
\begin{Proposition}[Vikulova~\cite{Vikulova_2022}]\label{Prop.3D}
Let $d=3$ and assume that the bounded domain~$\Omega$ is either convex 
or axiconvex with~$\partial\Omega$ homeomorphic to the sphere.
Then
$$
  \tau_D(\Omega) \geq \tau_{DN}(A_{R_1,R_2}) 
  \,
$$
where $\tau_{DN}(A_{R_1,R_2})$ is defined as in~\eqref{capacity.annulus}
but with $A_{R_1,R_2} := B_{R_2} \setminus \overline{B_{R_1}}$
being a spherical shell of radii
such that $|A_{R_1,R_2}| = |\Omega|$
and $|\partial B_{R_2}| = |\partial\Omega|$; explicitly,
$$
  R_1 := \frac{\sqrt[3]{|\partial\Omega|^{3/2}-6\pi^{1/2}|\Omega|}}{2\pi^{1/2}}
  \qquad \mbox{and} \qquad
  R_2 := \frac{|\partial\Omega|^{1/2}}{2\pi^{1/2}}
  \,.
$$
\end{Proposition}
 \begin{proof}[Proof of Theorem \ref{Thm.all} for $d=3$]
 
Proceeding as in the planar case,
to achieve Theorem~\ref{Thm.2D} via Lemma~\ref{Lem.core}
we use  the explicit formulation of $\tau_{DN}(A_{R_1,R_2})$
given in~\eqref{torsion_annulus},
and we compare it  with the torsional rigidity~$\tau_D(B_R)$
of a disk~$B_R$ such that $|\partial B_R| = |\partial\Omega|$, 
\ie\ $R=R_2$, see \eqref{ball}, obtaining 
\begin{equation}\label{rhs3D}
  \tau_\alpha(\Omega) - \tau_\alpha(B_R)
  \geq  
  \frac{4\pi}{45} \left(
  9 R_1^5 - 5 \frac{R_1^6}{R_2} 
  - 5 R_1^3 R_2^2 
  + R_2^5
  \right)
  - \frac{4\pi}{45} R^5
  +\frac{4\pi}{9\alpha} \left(
  \frac{(R_2^3-R_1^3)^2}{R_2^2}-R^4
  \right)
  ,
\end{equation}
where~$B_R$ is a ball such that $|\partial B_R| = |\partial\Omega|$
or $|B_R| = |\Omega|$, depending on the isoperimetric 
or isochoric constraint, respectively.
 
If we assume the  isoperimetric constraint then  $R=R_2$, so~\eqref{rhs3D} reduces to
\begin{equation}\label{rhs3Dper}
  \tau_\alpha(\Omega) - \tau_\alpha(B_R)
  \geq  
  \frac{4\pi}{45} R^5 \left(
  9 t^5 - 5 t^6 - 5t^3 + 5 \frac{(1-t^3)^2-1}{\alpha R}  
  \right)
\end{equation}
where we keep to denote $t := R_1 / R_2 \in [0,1)$.
Assuming $t \in (0,1)$
and requiring that the right-hand side is non-negative 
lead to the following smallness condition on~$\alpha$:
$$
  - \alpha R \leq 5 \ \frac{(1-t^3)^2-1}{9 t^5 - 5 t^6 - 5t^3}
  \,.
$$
Note that the denominator is non-zero (in fact, negative)
for every $t \in (0,1)$.
What is more, the right-hand side is bounded from below by~$2$.
Consequently, the right-hand side of~\eqref{rhs3Dper}
is non-negative provided that  
$ 
  \alpha \geq -\frac{2}{R} 
$.
In addition, the hypothesis of Lemma~\ref{Lem.core} requires
$
  \alpha > -\sigma_1(\Omega) 
$.
By the Weinstock inequality generalised to higher dimensions~\cite{BFNT-21}, 
$
  \sigma_1(\Omega) \leq \sigma_1(B_R) = \frac{1}{R}
$. 
This concludes the isoperimetric part of 
the proof of Theorem~\ref{Thm.all} if $d=3$.

In order to complete the proof, finally let us assume the isochoric constraint. 
In this case $R=\sqrt[3]{R_2^3-R_1^3}$, 
so~\eqref{rhs3D} reduces to
\begin{equation}\label{rhs3Dchor}
  \tau_\alpha(\Omega) - \tau_\alpha(B_R)
  \geq  
  \frac{4\pi}{45} R^5 \left(
  9 t^5 - 5 t^6 - 5t^3 + 1 
  - (1-t^3)^{5/3}
  + 5 \frac{(1-t^3)^2-(1-t^3)^{4/3}}{\alpha R}  
  \right)
\end{equation}
where we keep to denote $t := R_1 / R_2 \in [0,1)$.
Assuming $t \in (0,1)$
and requiring that the right-hand side is non-negative 
lead to the following smallness condition on~$\alpha$:  
$$ 
  - \alpha R \leq 5 \ \frac{(1-t^3)^{7/3}-(1-t^3)^{5/3}}
  {9 t^5 - 5 t^6 - 5t^3 + 1 - (1-t^3)^{5/3}}
  \,.
$$
Note that the denominator is non-zero (in fact, negative)
for every $t \in (0,1)$.
What is more, the right-hand side is bounded from below by~$1$.
Consequently, the right-hand side of~\eqref{rhs3Dper}
is non-negative provided that  
$ 
  \alpha \geq -\frac{1}{R} 
$.
In addition, the hypothesis of Lemma~\ref{Lem.core} requires
$
  \alpha > -\sigma_1(\Omega) 
$.
By the Weinstock inequality generalised 
to higher dimensions under the isochoric constraint~\cite{Brock_2001}, 
$
  \sigma_1(\Omega) \leq \sigma_1(B_R) = \frac{1}{R}
$. 
This concludes the isochoric part of 
the proof of Theorem~\ref{Thm.all} if $d=3$.
  \end{proof}
%--------------------------------------------------------%
\section{The general case of arbitrary dimensions}\label{Sec.all}
%--------------------------------------------------------%
% 
In this final section, we establish Theorem~\ref{Thm.all}
in all dimensions. 
To do that we need the following extension of Propositions~\ref{Prop.Payne} and~\ref{Prop.3D}
to all dimensions, which seems to be new for $d \geq 4$.

\begin{Theorem}\label{Prop.new}
Let $d \geq 1$ and assume that~$\Omega$ is a bounded convex domain.
Then
$$
  \tau_D(\Omega) \geq \tau_{DN}(A_{R_1,R_2}) 
  \,
$$
where $\tau_{DN}(A_{R_1,R_2})$ is defined as in~\eqref{capacity.annulus}
but with $A_{R_1,R_2} := B_{R_2} \setminus \overline{B_{R_1}}$
being a spherical shell of radii
such that $|A_{R_1,R_2}| = |\Omega|$
and $|\partial B_{R_2}| = |\partial\Omega|$. 
\end{Theorem}
\begin{proof}
The case $d=1$ is elementary. 
At the same time, Proposition~\ref{Prop.Payne} covers the two-dimensional case 
under a milder hypothesis. Therefore, we restrict to $d \geq 3$.
The case $d=3$ is kept in order to present an alternative proof
of Proposition~\ref{Prop.3D}.

If $d \geq 3$,
it is easy to verify that the solution~$u$ to~\eqref{capacity.annulus}
is given by 
\[
  u(x)=\displaystyle \frac{R_2^2-|x|^2}{2d}
  +\frac{R_1^d}{(d-2)d}\left(R_2^{2-d}-|x|^{2-d}\right)
  \,.
\] 
%Consequently,
%
%\begin{equation}\label{tors_shell}
 % \frac{\tau_{DN}(A_{R_1,R_2})}{|\Sphere^{d-1}|}
  %= \displaystyle 
%\dfrac{R_2^{d+2}}{d^2(d+2)}+\dfrac{R_1^{d+2}}{(d^2-4)}
%-\dfrac{R_1^d R_2^2}{d^2}-\dfrac{R_1^{2d}R_2^{2-d}}{d^2(d-2)}
 % \,.
%\end{equation}
%
Note that~$u$ is positive and radially symmetric.
Defining $\phi:[R_1,R_2] \to \Real$ by $u(x)=:\phi(|x|)$,
we see that~$\phi$ is decreasing and
\begin{align*}
   u_M:= \sup_{A_{R_1,R_2}} u =\phi(R_1)
   \,.
\end{align*}

We apply the method of parallel coordinates following~\cite{BFNT}. 
The idea is to construct a suitable trial function $\psi \in W_0^{1,2}(\Omega)$
in the variational characterisation of the Dirichlet torsional rigidity~\eqref{minimax},
which depends solely on the distance function from the boundary of~$\Omega$.
More specifically, we set
\begin{equation}\label{utest}
\psi(x):=
\begin{cases}
\phi\big(R_2-\rho(x)\big)\quad\ \ &\text{if} \quad  \rho(x)< R_2-R_1 \,,
\\
u_M\qquad\qquad&\text{if} \quad \rho(x)\geq R_2-R_1 \,,
\end{cases}
\end{equation}
where $\rho(x) $ is the distance function.
Obviously, $\psi=0$ on~$\partial\Omega$ and 
\begin{equation}\label{maxima}
  \psi_M:=\max_{\Omega} \psi \leq u_M
  \,.
\end{equation}
Moreover, for every $t \in [0,u_M]$, we have
\begin{gather*}
  |\nabla \psi|_{\psi=t} = |\nabla u|_{u=t} =: g(t) 
\end{gather*}

Define the following  sets:
%\noteD{True?}%
%
\begin{equation}\label{superlevelsets}
\begin{split}
	 E_{t}:= & \{ x\in\Omega:\ \psi(x)>t  \} \,, \\
	 A_{t}:= & \{x\in A_{R_1,R_2}:\ u(x)>t\} \,,\\
	 B_{t}:= & A_t\cup \overline{B_{R_1}} \,.
\end{split}
\end{equation}
By Lemma \ref{lem_der_per_e} 
and the Aleksandrov--Fenchel inequality~\eqref{iso_W_2},
we have, for almost every $t \in (0,\psi_M)$, 
\begin{equation}\label{Steiner}
-\dfrac{\der}{\der t} |\partial E_{t}|
\geq d(d-1)\dfrac{W_2(E_{t})}{g(t)}
\geq d(d-1)d^{-\frac{d-2}{d-1}}|B_1|^{\frac{1}{d-1}}
\dfrac{|\partial E_{t}|^{\frac{d-2}{d-1}}}{g(t)} \,.
\end{equation}
At the same time, for almost every $t \in (0,u_M)$,
\begin{equation}\label{int_per_ball}
	-\dfrac{\der}{\der t} |\partial B_{t}|
	=d(d-1)d^{-\frac{d-2}{d-1}} |B_1|^{\frac{1}{d-1}} 
	\dfrac{|\partial B_{t}|^{\frac{d-2}{d-1}}}{g(t)}
	\,.
\end{equation}
Because of the hypothesis $|\partial\Omega|=|\partial B_{R_2}|$, by \eqref{Steiner} and \eqref{int_per_ball}, integrating from $0$ to $t$, we obtain
\begin{equation}\label{perimetro_vivo}
  |\partial E_{t}| \leq |\partial B_{t}|, 
\end{equation} 
for every $t \in (0,\psi_M)$. 
Using this inequality together with~\eqref{maxima}
in the coarea formula,
we obtain
\begin{equation}
\label{gradient_estimates_e+}
  \int_{\Omega}|\nabla \psi|^2
  =  \int_{0}^{\psi_M} g(t)\, |\partial E_t| \, \der t 
  \leq \int_{0}^{u_M} g(t) \, |\partial B_{t}| \,\der t
  =\int_{A_{R_1,R_2}}|\nabla u|^2 
  \,.
\end{equation}

Now, we define $\mu(t):=|E_{t}|$ and $\eta(t):=|B_{t}|$. 
Using again the coarea formula and \eqref{perimetro_vivo}, 
we obtain 
\begin{equation*}
  \mu'(t)=-\int_{\{ u=t\}}\dfrac{1}{|\nabla \psi|} 
  =- \dfrac{|\partial E_{t}|}{g(t)}
  \geq -\dfrac{|\partial B_{t}|}{g(t)}
  =-\int_{\{ v=t\}}\dfrac{1}{|\nabla u|} 
  =\eta'(t)
\end{equation*}
for every $t \in (0\psi_M)$.
Since $\mu(0)=\eta(0)$ by the hypothesis $|\Omega|=|A_{R_1,R_2}|$, 
an integration from~$0$ to~$t$ yields
\begin{equation*}%\label{magmu}
\mu(t)\geq\eta(t)
\end{equation*}
for every $t \in [0,u_M)$.
Consequently,
%\noteD{Do we use \\ $E_t = \varnothing$ for $t>\psi_M$?}%
%
\begin{equation}\label{L_p_estimates_e+}
	\int_{\Omega} \psi
	= \int_{0}^{u_M} t \, \mu(t) \, \der t
	\geq
	\int_{0}^{u_M} t \, \eta(t)\, \der t
	= \int_{A_{R_1,R_2}} u
	\,.
\end{equation}

In summary,
using~\eqref{gradient_estimates_e+} and~\eqref{L_p_estimates_e+}
in the variational characterisation~\eqref{minimax}, we get 
\begin{equation*}
\tau_D( \Omega)\geq
\dfrac{\displaystyle\left(\int_{\Omega}\psi\right)^2}
{\displaystyle\int_{\Omega}|\nabla \psi|^2}
\geq 
\dfrac{\displaystyle\left(\int_{A_{R_1,R_2}} u\right)^2}
{\displaystyle\int_{A_{R_1,R_2}}|\nabla u|^2}=\tau_{DN}(A_{R_1,R_2})
\,.
\end{equation*} 
This concludes the proof of the theorem.
\end{proof}

In view of Theorem~\ref{Thm.all} already proved for $d=3$ 
in Section~\ref{Sec.3D},
it is enough to apply Theorem~\ref{Prop.new} for $d \geq 4$.
However, we allow $d=3$ in order to compare the present formulae
with those from Section~\ref{Sec.3D}.

Proceeding as in Sections~\ref{Sec.2D} and~\ref{Sec.3D},
a combination of Lemma~\ref{Lem.core} with Theorem~\ref{Prop.new} yields
\begin{multline} 
\label{s_p}
\tau_\alpha(\Omega) - \tau_\alpha(B_R)
  \\
  \geq   
  \dfrac{|B_1| R_2^{d+2}}{d(d+2)}
  +\dfrac{d |B_1| R_1^{d+2}}{(d^2-4)}
  -\dfrac{R_1^d R_2^2 |B_1|}{d}
  -\dfrac{R_1^{2d}R_2^{2-d}|B_1|}{d(d-2)}
  -\frac{|\Omega|^2}{|\alpha| |\partial\Omega|} 
  - \frac{|B_1| R^{d+2}}{d(d+2)}+\frac{|B_R|^2}{|\alpha| |\partial B_R|}
  \,,
\end{multline}
where~$B_R$ is a ball such that 
$|\partial B_R| = |\partial\Omega|$
or $|B_R| = |\Omega|$, depending on the isoperimetric 
or isochoric constraint, respectively.

\subsection{The isoperimetric constraint}
If $|\partial B_R|=|\partial \Omega|$, then $R_2=R$ and \eqref{s_p} becomes 
  \begin{gather}\label{rhsallDper}
\begin{split}
\tau_\alpha(\Omega) - \tau_\alpha(B_R)
  &\geq \dfrac{d|B_1| R_1^{d+2}}{(d^2-4)}
  -\dfrac{R_1^d R^2 |B_1|}{d}-\dfrac{R_1^{2d}R^{2-d}|B_1|}{d(d-2)}
  +\frac{|B_R|^2-|\Omega|^2}{|\alpha| |\partial\Omega|}
  \,, 
   \end{split} 
  \end{gather}
where the last term is positive by the isoperimetric inequality. 
Then, in order to guarantee that the right-hand side is non-negative, 
the following inequality has to hold
  \begin{equation}\label{n_dim}
  \frac{|B_R|^2-|\Omega|^2}{|\alpha| |\partial\Omega|}
  \geq \dfrac{|B_1| R_1^d R^2}{d(d^2-4)} \left((d^2-4)-d^2 t^2+(d+2)t^d\right)
  \end{equation}
where $t:=R_1/R\in [0,1)$.
We stress that the right-hand side is positive for all $t \in (0,1)$.
Under this condition, 
recalling that $|\Omega|= |B_1| (R^d-R_1^d)$  
and   $|\partial \Omega|= d |B_1| R^{d-1}$, 
we see that~\eqref{n_dim} holds if, and only if,
\begin{equation}\label{isoperimetric.all}
 -\alpha R \leq 
 \dfrac{(d^2-4)[1-(1-t^d)^2]}
 {t^d \,[(d^2-4)-d^2 t^2+(d+2)t^d]}.
\end{equation}
Since the right-hand side is bounded from below by~$2$,
the right-hand side of~\eqref{rhsallDper}
is non-negative provided that  
$ 
  \alpha \geq -\frac{2}{R} 
$.
In addition, the hypothesis of Lemma~\ref{Lem.core} requires
$
  \alpha > -\sigma_1(\Omega) 
$.
By the Weinstock inequality generalised to higher dimensions~\cite{BFNT-21}, 
$
  \sigma_1(\Omega) \leq \sigma_1(B_R) = \frac{1}{R}
$. 
This concludes the proof of Theorem~\ref{Thm.all}.

\subsection {The isochoric constraint}  
If $|B_R|=|\Omega|$, we have $R_2^d-R_1^d=R^d$. 
The right-hand side of~\eqref{s_p} is non-negative 
if the following inequality holds
\begin{equation}\label{any.iso} 
\begin{aligned}
\lefteqn{
  \frac{(d^2-4)R^{d+1}f(t)}{|\alpha|}
  =
  \frac{d(d^2-4)|B_R|^2\left(|\partial\Omega|-|\partial B_R|\right)}
  {|\alpha||B_1||\partial\Omega| |\partial B_R|}
  }
  \\
  &\ge (d^2-4)R_1^d R_2^2+(d+2)R_1^{2d}R_2^{2-d}+ (d-2) (R^{d+2}-R_2^{d+2})-d^2 R_1^{d+2}
  \\
  &= 
  R^{d+2}(1-t^d)^{-\frac{d+2}{d}} g(t)
  \,,
\end{aligned}  
\end{equation} 
where $t:=R_1/R_2 \in [0,1)$ and 
$$
\begin{aligned}
  f(t) &:= 1-(1-t^d)^{\frac{d-1}{d}} \,,
  \\
  g(t) &:= (d^2-4)t^d+(d+2)t^{2d}-d^2t^{d+2}
  -(d-2)\left[1- (1-t^d)^{\frac{d+2}{d}}\right] \,.
\end{aligned}
$$

\begin{Lemma}\label{Lem.horror}
Let $d \geq 3$.
For every $t \in (0,1)$, one has $g(t) >0$.
\end{Lemma}
\begin{proof}
It is straightforward to check that $g(0)=0=g(1)$.
Computing the derivative, we find 
\begin{equation}\label{gh}
        g'(t) = (d+2) t^{d-1} h(t),
    \end{equation}
 where
$$
  h(t) := d^2(1-t^2) - 2d (1-t^d) - (d-2) (1-t^d)^{2/d},
  \,
$$
with
$h(0)=d^2-3d+2>0$, for $d\geq 3$, and $h(1)=0$. 

Let us study the function $h$ in the interval $(0,1)$.
Computing the first and second derivatives, we find 
$$
  h'(0)=0 
  \,, \qquad 
  h''(0)=-2d^2
  \,, \qquad 
  \lim_{t \to 1^-} h'(t) = \lim_{t \to 1^-} h''(t) = + \infty
  \,.
$$
The key is to compute the third derivative
$$
  h'''(t) = 2 (d-1)(d-2)t^{d-3}
  \left[
  d^2 + (d-2) (1-t^d)^{2/d-3} (1+t^d)
  \right]
  .
$$
Since $h'''(t)>0$ for every $t \in (0,1)$, 
it follows that $h''$ is strictly increasing in $(0,1)$. 
Moreover, since $h''(0)<0$ and $h''(1)>0$, 
there exists a unique $t_1\in (0,1)$ such that $h''(t_1)=0$.
In summary, $h''<0$ in $(0,t_1)$, $h''(t_1)=0$ and $h''>0$ in $(t_1,1)$.
Consequently, $h'$ is strictly decreasing in $(0,t_1)$ 
and strictly increasing in $(t_1,1)$.
Now, considering that $h'(0)=0$ and $h'(1)>0$, 
there exists a unique $t_2 \in (0,1)$
(actually, we know that $t_2\in(t_1,1)$)
such that $h'(t_2)=0$, $h'<0$ in $(0,t_2)$ and $h'>0$ in $(t_2,1)$. 
This means that $h$ is strictly decreasing in $[0,t_2]$
and strictly increasing in $[t_2,1]$.
Since $h(1)=0$, necessarily 
$$
  h(t_2)=  \min_{[0,1]} h <0 
  \,.
$$
Recalling that $h(0)>0$, 
there exists a unique point $t_3\in (0,1)$ such that 
$h>0$ in $(0,t_3)$, $h(t_3) = 0$ and $h<0$ in $(t_3,1)$.

By the relationship~\eqref{gh} 
and the previous analysis of~$h$,
we conclude that 
\begin{equation}\label{g'}
  g'>0 \mbox{ in } (0,t_3)
  \quad \mbox{and} \quad
  g'<0 \mbox{ in } (t_3,1),
  \quad \mbox{while} \quad g(1)=g(0)=0.
\end{equation}
This implies the desired claim that~$g$ is positive in $(0,1)$.
\end{proof}

By virtue of the lemma, the right-hand side of~\eqref{any.iso} 
is positive for every $t \in (0,1)$. 
Under this condition, the inequality of~\eqref{any.iso} holds 
provided that  
\begin{equation}\label{iso_d}
-\alpha R \le \displaystyle \frac{(d^2-4)\left[(1-t^d)^{\frac{d+2}{d}}-(1-t^d)^{\frac{2d+1}{d}}\right]}{g(t)}
\,.
\end{equation}
\begin{Lemma}\label{Lem.key}
Let $d \geq 3$.
For every $t \in (0,1)$, the right-hand side of~\eqref{iso_d}
is bounded from below by~$1$.
\end{Lemma}
\begin{proof}
The claim is equivalent to the property that 
$$
  k(t) :=
(d^2-4)\left[(1-t^d)^{1+\frac{2}{d}}-(1-t^d)^{2+\frac{1}{d}} \right]-  g(t) \geq 0.
$$
Since $k(0)=k(1)=0$, in order to prove the claim, 
it is enough to show that~$k$ has a unique extremal point 
and that this point is a maximum. 
We shall proceed as in the previous Lemma~\ref{Lem.horror} 
by proving that~\eqref{g'} holds also for~$k$,
with some (possibly different) $t_3 \in (0,1)$.

Computing the derivative, we find
\begin{equation}\label{deriv_k}
    k'(t)=(d+2) t^{d-1}m(t),
\end{equation}
where
\begin{equation*}
    m(t):=(d-2)(2d+1)(1-t^d)^{1+\frac 1 d}-(d-2)(d+1)(1-t^d)^{\frac 2 d}-d^2(1-t^2)+2d(1-t^d).
\end{equation*}
Note that $m(0)=m(1)=0$.
In order to simplify the computations, 
we define the change of variable $u:=1-t^d\in(0,1)$
leading to  
\begin{equation}\label{M}
    M(u):=m(t(u))= d^2(1-u)^{\frac{2}{d}}-(d-2)(d+1)u^{\frac{2}{d}}+(2d^2-3d-2)u^{1+\frac{1}{d}}+2d u -d^2.
\end{equation}
By computing the first and second derivatives, we find
\begin{equation} \label{ext_M}
  \begin{aligned}
  \lim_{u\to0^+}M'(u)=-\infty, 
  \\
  \lim_{u\to1^-}M'(u) =-\infty,
  \end{aligned}
  \begin{aligned}
  \lim_{u\to0^+}M''(u) &=+\infty, 
  \\
  \qquad\lim_{u\to1^-}M''(u) &=-\infty.
  \end{aligned}
\end{equation}
As in Lemma~\ref{Lem.horror}, the key is to compute the third derivative
$$
  M'''(u)=
  -\frac{(d-1)(d-2)}{d^3} 
  \left[
  4 d^2 (1-u)^{\frac{2}{d}-3}
  + (d+1) (2d+1) u^{\frac{1}{d}-2}
  + 4 (d+1)(d-2) u^{\frac{2}{d}-3}
  \right] .
$$ 
Since $M'''(u) < 0$ for every $u \in (0,1)$,
$M''$ is is strictly decreasing in $(0,1)$.
Recalling the boundary asymptotics for~$M''$ from~\eqref{ext_M},
we deduce that there exists a unique point $u_0\in (0,1)$ 
such that $M''(u_0)=0$. 
Consequently, $u_0\in (0,1)$ is the unique interior maximum point of $M'$. 
Since~$M$ is non-constant and $M(0) = 0 = M(1)$, 
it is impossible that $M'(u) \leq 0$ for every $u \in (0,1)$.
Hence, $M'(u_0)>0$.
Taking also into account the boundary asymptotics for~$M'$ from~\eqref{ext_M},
there exist exactly two points $u_1<u_2$ in $(0,1)$ 
such that $M'(u_1)=M'(u_2)=0$.
In summary, 
\begin{equation*}
    \begin{cases}
        M'<0 \;&\text{in}\quad(0,u_1), \\
        M'>0\;&\text{in}\quad(u_1,u_2), \\
       M'<0 \;&\text{in}\quad(u_2,1).
         \end{cases}
\end{equation*}
This means  that $u_1$ and $u_2$ are the two unique extremal points of~$M$  
satisfying $u_1<u_2$ and 
$$ M(u_1)=\min_{[0,1]}M<0, \qquad M(u_2)=\max_{[0,1]}M>0.$$ It follows that there exists a unique $u_3\in (0,1)$  such that $M(u_3)=0$ 
(actually, we know that $u_3 \in (u_1,u_2)$).
Recalling the definition of the function $M$ in \eqref{M}, 
we have proved that there exists a unique $t^*\in (0,1)$ such that 
\begin{equation*}
    \begin{cases}
        m(t)>0 \;&\text{in}\quad (0,t^*), \\
    m(t^*)=0, \\
       m(t)<0 \;&\text{in}\quad (t^*,1).
         \end{cases}
\end{equation*}

By \eqref{deriv_k}, we have that $t^*\in (0,1)$ is the unique extremal point of $k$ and it such that 
$$  k(t^*)=\max_{(0,1)}k>0. $$
Finally, by the fact that $k(0)=0=k(1)$, 
we conclude that $k>0$ in $(0,1)$.
\end{proof} 

The lemma implies that~\eqref{any.iso} holds provided that 
$ 
  \alpha \geq -\frac{2}{R} 
$.
In addition, the hypothesis of Lemma~\ref{Lem.core} requires
$
  \alpha > -\sigma_1(\Omega) 
$.
By the Weinstock inequality generalised 
to higher dimensions under the isochoric constraint~\cite{Brock_2001}, 
$
  \sigma_1(\Omega) \leq \sigma_1(B_R) = \frac{1}{R}
$. 
This concludes the isochoric part of the proof of Theorem~\ref{Thm.all}.

\begin{Remark}
Given any positive $R_1 < R_2$,
both the conditions~\eqref{isoperimetric.all} and~\eqref{iso_d}  
are equivalent to 
$$
  -\alpha R \leq 
  \frac{\displaystyle
  \frac{|B_R|^2}{|\partial B_R|} - \frac{|\Omega|^2}{|\partial\Omega|}}
  {\tau_D(B_R) - \tau_{DN}(A_{R_1,R_2})} \, R
  \,,
$$
where $|\Omega| := |A_{R_1,R_2}|$ and $|\partial\Omega| := |\partial B_{R_2}|$
(\ie, the numbers~$|\Omega|$ and $|\partial\Omega|$ 
are uniquely determined by the given radii $R_1,R_2$,
the initial problem of a set~$\Omega$ is disregarded now).
The isoperimetric (respectively, isochoric) constraint means
$|\partial B_R| = |\partial B_{R_2}|$ 
(respectively, $|B_R| = |A_{R_1,R_2}|$),
which corresponds to the choice 
$R:=R_2$ (respectively, $R:=(R_2^d-R_1^d)^{1/d}$).
The fact that the right-hand side of~\eqref{isoperimetric.all} 
(respectively, of~\eqref{iso_d}) is bounded from below by~$2$
(respectively, by~$1$, see Lemma~\ref{Lem.key}) 
establishes a non-trivial quantitative estimate
\begin{equation}\label{quantitative}
  \tau_D(B_R) - \tau_{DN}(A_{R_1,R_2})
  \leq 
  c \, R
  \left(
  \frac{|B_R|^2}{|\partial B_R|} - \frac{|\Omega|^2}{|\partial\Omega|}
  \right)
  ,
\end{equation}
where $c:=1/2$ or $c:=1$ 
in the isoperimetric or isochoric constraint, respectively. 
Note that both sides of~\eqref{quantitative} are positive.
\end{Remark}

\subsection*{Acknowledgement}
N.~Gavitone and G.~Paoli  were partially supported by Gruppo Nazionale per l'Analisi Matematica, la Probabilità e le loro Applicazioni
(GNAMPA) of Istituto Nazionale di Alta Matematica (INdAM).  

N.~Gavitone was supported by the Project MUR PRIN-PNRR 2022: ``Linear and
Nonlinear PDE'S: New directions and Applications'', 
P2022YFA and by GNAMPA group of INdAM.

G.~Paoli was supported by ``INdAM - GNAMPA Project", codice CUP E5324001950001
    and  by the Project PRIN 2022 PNRR:  ``A sustainable and trusted Transfer Learning platform for Edge Intelligence (STRUDEL)", CUP E53D23016390001, in the framework of European Union - Next Generation EU program

D.~Krej\v{c}i\v{r}{\'\i}k was supported
by the grant no.~26-21940S
of the Czech Science Foundation.

%\newpage
%\vfill  
%--------------%
% BIBLIOGRAPHY %
%--------------%
%
%\addcontentsline{toc}{section}{References}
\bibliographystyle{amsplain}
\bibliography{biblio14}

\end{document}